\documentclass{amsart}

\usepackage{BJP}
\newcommand{\asep}{.3}
\newcommand{\step}{1.3}
\newcommand{\bsep}{.35}

\title{Extremal stability for configuration spaces}

\author{Ben Knudsen}\thanks{Ben Knudsen was supported in part by NSF grant DMS-1906174}
\address{Department of Mathematics, Northeastern University, USA}
\email{b.knudsen@northeastern.edu}

\author{Jeremy Miller}\thanks{Jeremy Miller was supported in part by a Simons collaboration grant}
\address{Department of Mathematics, Purdue University, USA}
\email{jeremykmiller@purdue.edu}

\author{Philip Tosteson} \thanks{Philip Tosteson was supported in part by NSF grant  DMS-1903040.}
\address{Department of Mathematics, University of Chicago, Chicago, IL}
\email{{ptoste@math.uchicago.edu}}

 \date{\today}

\begin{document}

\begin{abstract} We study stability patterns in the high dimensional rational homology of unordered configuration spaces of manifolds. Our results follow from a general approach to stability phenomena in the homology of Lie algebras, which may be of independent interest.
\end{abstract}
\maketitle

\section{Introduction}

The purpose of this paper is to investigate a stability phenomenon in the high dimensional rational homology of the unordered configuration spaces of a $d$-manifold $M$.\footnote{Throughout, we use the shorthand ``manifold'' to refer to a topological space homeomorphic to the interior of a compact topological manifold with boundary. Since configuration spaces are isotopy invariant, our results also hold for manifolds with non-empty boundary, \emph{mutatis mutandis}.}  Writing $$\Conf_n(M) = \{(x_1,\ldots,x_n)\in M^n \, |\, x_i \neq x_j \text{ for } i \neq j \}$$ for the ordered configuration space of $n$ points in $M$ and $B_n(M)=\Conf_n(M)/S_n$ for the unordered configuration space, classical results of McDuff, Segal, and Church show that the spaces $B_n(M)$ exhibit \emph{rational homological stability}. We will restrict our attention to the case of even $d$, since there is a simple closed form available for $d$ odd \cite{BCT} (see Section \ref{section:odd case} for further commentary).

\begin{theorem}[Homological stability \cite{Mc1,Se,Ch}]\label{thm:homological stability}
Let $M$ be a manifold of even dimension $d\geq 2$. For each $i\geq0$, there is a polynomial in $n$ of degree at most $\dim H_0(M;\mathbb{Q})-1$, which coincides with $\dim H_i(B_n(M);\mathbb{Q})$ for all $n$ sufficiently large.
\end{theorem}

Traditionally, Theorem \ref{thm:homological stability} is stated in the case that $M$ is connected, where it is the statement that the function $n\mapsto \dim H_i(B_n(M);\mathbb{Q})$ is eventually constant. The general case follows easily from the connected case and the K\"{u}nneth theorem. 

While homological stability is a stability pattern in low homological dimension, extremal stability, our main theorem, is a pattern in low homological \emph{codimension}. Since $H_i(B_n(M))=0$ for $i>\nu_n:=n(d-1)+|\pi_0(M)|$, it is reasonable to think of $H_{\nu_n-i}(B_n(M))$ as the codimension $i$ homology of $B_n(M)$, at least generically.

\begin{theorem}[Extremal stability]\label{thm:extremal stability}
Let $M$ be a manifold of even dimension $d\geq 2$. For each $i\geq0$, there is a quasi-polynomial in $n$ of degree at most $\dim H_{d-1}(M;\mathbb{Q}^w)-1$ and period at most $2$, which coincides with $\dim H_{\nu_n-i}(B_n(M);\mathbb{Q})$ for all $n$ sufficiently large.
\end{theorem}

%\begin{figure}[!ht] \begin{center}\scalebox{.4}{\includegraphics{Extremalpicture.png}}\end{center}
%\label{MSpaint}
%\caption{Homological stability and extremal stability}
%\end{figure}  
%\note{Jeremy: I got a few MS paint pictures into G\& T. Still, it probably would be good to replace this with something more professional. Ben: Agreed, sadly I have no skills}

Here, we have written $\mathbb{Q}^w$ for the orientation sheaf of $M$.  Equivalently, Theorem \ref{thm:extremal stability} states that there are two polynomials $p_{{\rm even}}^i(n)$ and $p_{\rm odd}^i(n)$ governing the codimension $i$ homology of $B_n(M)$ for even and odd $n$, respectively. See \S \ref{section:modules and growth} for our conventions on quasi-polynomials.   

\begin{example}[{\cite{DK}}]\label{example:DK}
For $\Sigma$ a compact, orientable surface of genus $2$ and $n\geq4$, it follows from \cite[Corollary 4.9 and Figure 1]{DK} that \[\dim H_{\nu_n}(B_n(\Sigma);\mathbb{Q})=\begin{cases}
\frac{n^3+4n^2-4n-16}{16}\vspace{.1cm} &\quad n\text{ even}\\ 
\frac{n^3+n^2-n+15}{16}&\quad n \text{ odd}.
\end{cases}\]
This example shows that the degree bound of Theorem \ref{thm:extremal stability} is sharp.
\end{example}

\extremalfigure

Instances of extremal stability for configuration spaces were first noticed by Maguire in the case $M=\bbC P^3$ \cite{Megan}, who proved a stability theorem for manifolds with even-dimensional cohomology. The only other prior example of extremal stability that we are aware of is in the high dimensional cohomology groups of congruence subgroups of $SL_n(\Z)$ \cite{MNP}.

If $H_{d-1}(M, \mathbb{Q}^w ) = 0$, then Theorem \ref{thm:extremal stability} is just the statement that the groups $H_{\nu_n-i}(B_{n}(M), \bbQ)$ eventually vanish. In this situation, $H_{\nu_n-i}(B_{n}(M), \bbQ) $ should not be thought of as the codimension $i$ homology of $B_n(M)$. In Theorem \ref{thm:vanishing}, we prove a more refined stability result for manifolds with vanishing high degree homology groups.

\begin{remark}
As observed by the referee, a quasi-polynomial such as the one referenced in Theorem \ref{thm:extremal stability} may be written in the form $p(n)+(-1)^nq(n)$ for polynomials $p$ and $q$, and, since the polynomials of Example \ref{example:DK} have the same leading term, we have $\deg q(n)<3$ in this example. We do not know whether this improved bound on $q(n)$ is a general phenomenon.  
\end{remark}

\subsection{Stabilization maps} \label{stabmap}

Extremal stability is induced by a family of maps of the form $$H_i(B_n(M);\Q) \to H_{i+2d-2}(B_{n+2}(M);\Q).$$ Heuristically, such a map is obtained as follows (we assume that $M$ is orientable for simplicity). Fixing a class $\alpha \in H_{d-1}(M;\mathbb{Q})$, which we imagine as a $(d-1)$-parameter family of configurations of a single point in $M$, we write $\alpha \otimes [v,v] \in H_{2d-2}(B_2(M);\Q)$ for the class obtained by replacing this single point with a pair of orbiting points. Here, we think of $v$ as the class of a single point in $\mathbb{R}^d$ and $[\cdot ,\cdot ]$ as the Browder bracket in the homology of Euclidean configuration spaces (nontrivial for $d$ even).

We wish to define our stabilization map by superposition with the class $\alpha \otimes [v,v]$; that is, we attempt to send the cycle depicted on the left of Figure \ref{StabilizationPic} to the cycle depicted on the right. Unfortunately, this cycle does not lie in the configuration space, since it contains configurations of non-distinct particles. 

\begin{figure}[!ht] \begin{center}\scalebox{.2}{\includegraphics{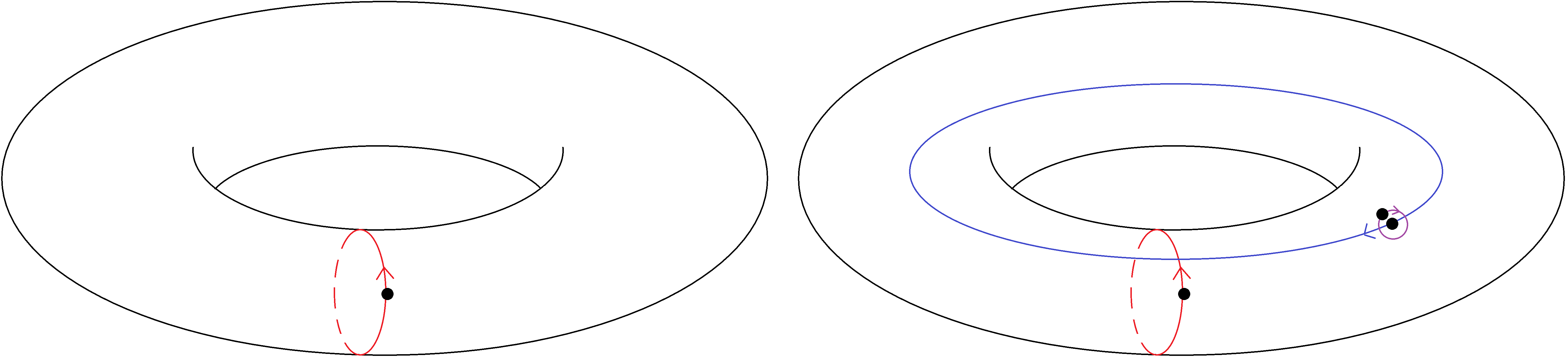}}\end{center}

\caption{Extremal stabilization map}

\label{StabilizationPic}
\end{figure}  

Heuristically, this issue can be resolved in the example depicted (where $i=1$) by performing surgery in a Euclidean neighborhood of the intersection point, resulting in a $(2d-1)$-chain in $B_3(M)$ whose boundary corresponds to the image of the nested Browder bracket $[v,[v,v]] \in H_{2d-2}(B_3(\R^d);\Q)$ under the coordinate embedding $\mathbb{R}^d\subseteq M$. Since $[v,[v,v]]=0$ by the Jacobi relation, we may choose a bounding chain, removing the point of intersection.

We now sketch a more categorical approach to this idea. The desired map is a special case of a map 
 \[H_i(B_n(M)) \otimes H_j(B_m(M)) \to H_{n(d-1)+i+j}(B_{2n+m}(M))\] that replaces points in $B_n(M)$ with pairs of orbiting points and takes the union with the configuration in $B_m(M)$. Such a map can be obtained by applying factorization homology \cite{AFTop} to an appropriate chain level map of $E_d$-algebras in the case $M=\mathbb{R}^d$. Setting $B(\mathbb{R}^d)=\bigsqcup_{n\geq0} B_n(\mathbb{R}^d)$, the tensor product $C_*(B(\mathbb{R}^d))\otimes C_*(B(\mathbb{R}^d))$ is obtained from the coproduct of $E_d$-algebras by attaching an $E_d$-cell along the Browder bracket of the respective generators. Universal properties of free algebras and colimits imply that it is enough to choose two cycles, in our case $v$ and $[v,v]$, together with a nullhomotopy of their Browder bracket. Since $[[v,v],v]=0$, a nullhomotopy exists.
 
We pursue neither of these sketches here. Instead, we notice that both rely on the same key fact, namely that $[v,v]$ lies in the center of the shifted Lie algebra $H_*(B(\mathbb{R}^d))$. This algebraic observation points the way to a simple and rigorous definition of the desired extremal stabilization maps.

\subsection{Transit algebras} 

We approach Theorem \ref{thm:extremal stability}, a priori a topological statement, in an entirely algebraic setting. As shown by the first author \cite{Knudsen}, the total homology of the configuration spaces of a manifold is the homology of a certain Lie algebra---see Theorem \ref{thm:lie homology} below---so stability phenomena in Lie algebra homology give rise to stability phenomena in the homology of configuration spaces. 

The homology of a Lie algebra $\fg$ supports two types of action relevant to stability. First, given an Abelian quotient $\fg \to \fk$, there results an action of $\Sym(\fk^\vee[-1])$ on the Lie algebra homology of $\mathfrak{g}$; in the example of configuration spaces, this degree lowering action gives rise to the maps \[H_*(B_n(M), \bbQ) \to H_*(B_{n-1}(M), \bbQ)\] obtained from the inclusion ${\rm Conf}_{n}(M) \subseteq {\rm Conf}_{n-1}(M) \times M$  by pairing with an element of $H^*(M)$ and applying transfer. Second, given a central subalgebra $\fh \subseteq \fg$, there results an action of $\Sym(\fh[1])$ on the Lie algebra homology of $\mathfrak{g}$; in the example of configuration spaces, this degree raising action gives rise to the extremal stabilization maps described heuristically above. For details on the relevant Lie algebras, the reader may consult Section \ref{section:applications}.

These two actions do not commute, so there does not result an action of the tensor product of the two algebras. Instead, the two actions interact via Weyl relations, giving rise to an action of an algebra that we call a \emph{transit algebra}---see Definition \ref{def:transit algebra}. As shown in \S\ref{section:applications}, transit algebras provide a common source for classical homological stability and extremal stability.

\subsection{Acknowledgments}
We thank Megan Maguire for helpful conversations and for the initial inspiration to investigate this question. We thank Peter Patzt and Muhammad Yameen for helpful comments on an earlier version of this paper. We thank the referee for their careful reading and constructive feedback.

\subsection{Conflict of interest}
On behalf of all authors, the corresponding author states that there is no conflict of interest.

\section{Algebraic background}

In this section, we detail our conventions on graded objects, algebras, and modules.

\subsection{Degrees and slopes} We work in the setting of bigraded vector spaces over $\mathbb{Q}$. Such a vector space $V$ arrives equipped with a decomposition $V=\bigoplus_{n,i\in\mathbb{Z}} V_{n,i}$. We write $\langle r\rangle$ and $[s]$ for the relevant shift operations, i.e., \[(V\langle r\rangle [s])_{n,i}=V_{n-r,i-s}.\] The parameters $n$ and $i$ are called the weight and the (homological) degree, respectively, and we write $w(v)=n$ and $d(v)=i$ for a bihomogeneous element $v\in V$ of bidegree $(n,i)$.  

Duals and tensor products of bigraded vector spaces are defined by the stipulations \[(V^\vee)_{n,i}=\mathrm{Hom}_{\mathbb{Q}}(V_{-n,-i},\mathbb{Q})\qquad\qquad (V\otimes W)_{n,i}=\bigoplus_{a+b=n,c+d=i}V_{a,c}\otimes W_{b,d}.\] With this tensor product, the category of bigraded vector spaces is monoidal, and we equip it with the symmetric monoidal structure whose symmetry incorporates Koszul signs in the homological degree \emph{but not in the weight.} Because of this symmetry, parity of degree will play an important role in what follows, and we write \[V^\epsilon=\bigoplus_{n\in\mathbb{Z}}\,\bigoplus_{i\equiv \epsilon\,\mathrm{mod}\,2}V_{n,i}\] for $\epsilon\in\{0,1\}$, considered as a bigraded subspace of $V$.

We say that $V$ is of finite type if $V_{n,i}$ is finite dimensional for $n,i\in\mathbb{Z}$. If $V$ is of finite type and vanishes in a cofinite set of bidegrees, we say that $V$ is finite dimensional. We say that $V$ is bounded below if $V_{n,i}=0$ whenever either $n$ or $i$ is a negative number of sufficiently large absolute value. We say that $V$ is connected if $V_{0,0}=0$. 

We record the following simple observations regarding the interaction of these finiteness properties with duals and tensor products.

\begin{lemma}\label{lem:finiteness}
Let $V$ and $W$ be bigraded vector spaces.
\begin{enumerate}
\item If $V$ is of finite type or finite dimensional, then so is $V^\vee$. 
\item If $V$ and $W$ are finite dimensional, then so is $V\otimes W$.
\item If $V$ and $W$ are bounded below, then so is $V\otimes W$.
\item If $V$ and $W$ are bounded below of finite type, then so is $V\otimes W$.
\end{enumerate}
\end{lemma}

We think of a bigraded vector space as a planar grid of vector spaces, with the weight recorded on the horizontal axis and the degree on the vertical. Our language reflects this idea; for example, we say that $V$ is first-quadrant if $V_{n,i}=0$ for $n<0$ and $i<0$ (in which case $V^\vee$ is third-quadrant). This picture also informs the following.

\begin{definition}
Fix a bigraded vector space $V$.
\begin{enumerate}
\item Let $0\neq v\in V$ be a bihomogeneous element of nonzero weight. The \emph{slope} of $v$ is the rational number $m(v)=d(v)/w(v)$.  
\item Suppose that $V\neq0$ is concentrated in nonzero weights. The \emph{maximal slope} of $V$ is \[m_{\max}(V)=\max \{C\in\mathbb{Q}\mid\exists v\in V : m(v)=C\},\] provided the maximum exists (resp. minimal slope, $m_{\min}(V)$, $\min$, minimum).
\end{enumerate}
\end{definition}

Given $C\in\mathbb{Q}$, we write $V_C\subseteq V$ for the span of the bihomogeneous elements of slope $C$, and similarly for $V_{<C}$, $V_{>C}$, and so on. If $V=V_C$, then we say that $V$ is of slope $C$. Note that, if $V$ is of slope $C$, then $V$ is concentrated in nonzero weight.

A related notion of slope will also be important in what follows.

\begin{definition}
Fix a bigraded vector space $V$ and $C\in\mathbb{Q}$. A \emph{ray of slope} $C$ in $V$ is a subspace of the form $\mathcal{R}=\bigoplus_{u\geq0} V_{au, bu+i_0}$, where $C=b/a$ with $(a,b)=1$. The \emph{graded dimension} of $\mathcal{R}$ is the function \[n\mapsto\begin{cases}
 \dim \mathcal{R}_{n,Cn+i_0}&\quad a\mid n\\
 0&\quad \text{otherwise}. 
 \end{cases}\] 
\end{definition}

At times, it will be convenient to work with a third grading, by polynomial degree, which will typically be either non-negative or non-positive. Bigraded vector spaces may be regarded as trigraded vector spaces concentrated in polynomial degree $1$; trigraded vector spaces may be regarded as bigraded vector spaces by forgetting the third grading. Tensor products and duals of trigraded vector spaces are defined analogously, with Koszul signs reflecting only the homological degree. Discussions of slope will never involve the polynomial grading.

\begin{example}\label{example:degrees}
Let $\mathbb{Q}(n,i,r)$ denote the trigraded vector space that is nonzero only in tridegree $(n,i,r)$, where it is $1$-dimensional. There is a canonical isomorphism \[\mathbb{Q}(n,i,r)\otimes\mathbb{Q}(m,j,s)\cong\mathbb{Q}(m+n,i+j,r+s)\] of trigraded vector spaces, under which the monoidal symmetry acts by $(-1)^{ij}$. Furthermore, there is a canonical isomorphism \[\mathbb{Q}(n,i,r)^\vee\cong\mathbb{Q}(-n,-i,-r).\]
\end{example}

%As the example shows, the evaluation pairing $V\otimes V^\vee\to\mathbb{Q}$ respects the first two gradings but not the third. This curiosity will play no role in the ensuing discussion.

\subsection{Symmetric (co)algebras and their duals} Given a trigraded vector space $V$, we write $\Sym^k(V)=(V^{\otimes k})_{S_k}$, where the symmetric group acts via the symmetry of the symmetric monoidal structure introduced above. 

\begin{definition}
Let $V$ be a trigraded vector space. The \emph{symmetric algebra} on $V$ is the trigraded vector space $\Sym(V)=\bigoplus_{k\geq0}\Sym^k(V)$, equipped with the product given componentwise by the dashed filler in the commuting diagram \[\xymatrix{
V^{\otimes k}\otimes V^{\otimes \ell}\ar[d]\ar@{=}[r]^-\sim& V^{\otimes (k+\ell)}\ar[d]\\
\Sym^k(V)\otimes \Sym^\ell(V)\ar@{-->}[r]&\Sym^{k+\ell}(V),
}\] where the vertical arrows are the respective projections to the coinvariants.
\end{definition}

This multiplication map furnishes the symmetric algebra with the structure of a commutative algebra object in the category of trigraded vector spaces. It is a universal commutative algebra in the sense of the commutative diagram \[\xymatrix{
V\ar[r]\ar[d]&\mathcal{A}\\
\Sym(V)\ar@{-->}[ur]_{\exists !}
}\] of trigraded vector spaces, in which $\mathcal{A}$ is a trigraded commutative algebra and the dashed filler a map of such. In particular, there is a canonical, natural map of algebras \[\Sym(V\oplus W)\to \Sym(V)\otimes \Sym(W)\] induced by the assignment $(v,w)\mapsto v\otimes 1+1\otimes w$, which is easily seen to be an isomorphism. This map furnishes the symmetric algebra functor with an oplax monoidal structure; in particular, $\Sym(V)$ is canonically a bicommutative bialgebra with comultiplication induced by the diagonal of $V$. 

Given a trihomogeneous basis $\{t_i\}_{i\in I}$ for $V$, a trihomogeneous basis for $\Sym^k(V)$ is provided by the set of equivalence classes of degree $k$ monomials $t_{i_1}\cdots t_{i_k}$ under the equivalence relation generated by the relation \[t_{i_1}\cdots t_{i_j}t_{i_{j+1}}\cdots t_{i_k}\sim (-1)^{d(t_{i_j})d(t_{i_{j+1}})}t_{i_1}\cdots t_{i_{j+1}}t_{i_{j}}\cdots t_{i_k}.\] In this way, we obtain a trihomogeneous basis for $\Sym(V)$, which we refer to as the \emph{monomial basis}. In this basis, multiplication is given by concatenation of monomials and comultiplication by the usual shuffle coproduct. It should be emphasized that a generator behaves as an exterior or polynomial generator according to the parity of its homological degree alone.

In the presence of the above basis, we write $\partial_{t_i}\in V^\vee$ for the functional dual to $t_i$, i.e., $\partial_{t_i}(t_j)=\delta_{ij}.$ We extend these functionals to endomorphisms of $\Sym(V)$ by requiring them to act as derivations, i.e., \[\partial_t(pq)=\partial_t(p)q+(-1)^{d(p)}p\partial_t(q)\] for monomials $p$ and $q$. In other words, the functional $\partial_t$ acts according to the usual rules of differential calculus, with appropriate Koszul signs. Note that the resulting endomorphism of $\Sym(V)$ lowers each degree by the corresponding degree of $t$.

When $V$ is of finite type, the functionals $\partial_{t_i}$ generate $V^{\vee}$.  In this case, since $\partial_{t_1}\partial_{t_2}=(-1)^{d(t_1)d(t_2)}\partial_{t_2}\partial_{t_1}$, there is a well-defined endomorphism of $\Sym(V)$ for every monomial in $\Sym(V^\vee)$. In light of the preceding observation on the tridegree of $\partial_t$, this construction determines a trigraded linear embedding \[\Sym(V^\vee)\to \Sym(V)^\vee,\] which is an isomorphism under appropriate finiteness assumptions; in particular, the isomorphism holds if $V$ is a bigraded vector space bounded below of finite type, regarded as concentrated in polynomial degree $1$. Since this fact is essentially standard, we content ourselves with reminding the reader that symmetric group invariants and coinvariants are canonically isomorphic in characteristic zero, and that our conventions on trigraded duals ensure that the target of the map in question is an infinite sum rather than an infinite product.

In this situation, by adjunction, there is a one-to-one correspondence of sets of maps \[\varphi:\Sym(V^\vee)\otimes \mathcal{M}\to\mathcal{M} \qquad \iff\qquad \tilde \varphi:\mathcal{M}\to \Sym(V)\otimes\mathcal{M}\] for any bounded below trigraded vector space $\mathcal{M}$, which is given explicitly by the formulas $\tilde \varphi(m)=\sum_p p\otimes \varphi(\partial_p\otimes m)$ and $\varphi(\partial_p\otimes m)=\partial_p\cap \tilde \varphi(m),$ where $p$ ranges over the monomial basis (finite in each tridegree). In this way, a $\Sym(V^\vee)$-module structure on $\mathcal{M}$ determines a $\Sym(V)$-comodule structure, referred to as the \emph{adjoint comodule}. If $\mathcal{M}$ is finite dimensional in each tridegree, then this comodule structure is in turn equivalent to the $\Sym(V^\vee)$-module structure on $\mathcal{M}^\vee$ determined by the map \[\Sym(V^\vee)\otimes \mathcal{M}^\vee\cong (\Sym(V)\otimes\mathcal{M})^\vee\xrightarrow{\tilde\varphi^\vee}\mathcal{M}^\vee.\] We refer to this module as the \emph{dual adjoint module}.

\subsection{Modules and growth}\label{section:modules and growth} In what follows, we will be interested in the eventual growth rates of graded dimensions of vector spaces.

\begin{definition}
A \emph{quasi-polynomial} is an element of $\Pi[t]$, where $\Pi$ is the ring of periodic functions from $\mathbb{Z}$ to $\mathbb{Q}$. The \emph{period} of a quasi-polynomial is the least common multiple of the periods of its coefficients.
\end{definition}

When working over a symmetric algebra with first-quadrant generators of fixed slope, rays of this slope exhibit predictable growth. As a matter of notation, we write \[\mathrm{lcm}(V)=\mathrm{lcm}\{n\in\mathbb{Z}\mid \exists i\in\mathbb{Z}:V_{n,i}\neq0\},\] where $V$ is assumed to vanish in a cofinite set of degrees including zero.

\begin{lemma}\label{lem:finitely generated growth}
Let $V$ be a finite dimensional, first-quadrant bigraded vector space of slope $C\in\mathbb{Q}$ and $\mathcal{M}$ a finitely generated $\Sym(V)$-module. The graded dimension of any ray in $\mathcal{M}$ of slope $C$ is eventually equal to a quasi-polynomial of degree at most $\dim V^0-1$ and period dividing $\mathrm{lcm}(V^0)$.
\end{lemma}
\begin{proof}
Any ray in $\mathcal{M}$ is a $\Sym(V)$-submodule, since $V$ has slope $C$, and $\Sym(V)$ is Noetherian, since $V$ is finite dimensional, so we assume that $\mathcal{M}$ is itself a ray of slope $C$. Applying the shift functor $[r]$ to this ray does not change its graded dimension, so we may assume that $\mathcal{M}=\mathcal{M}_C$. Lastly, we may assume that $V=V^0$; indeed, $\Sym(V)\cong\Sym(V^0)\otimes\Sym(V^1)$, and $\Sym(V^1)$ is finite dimensional, since $V^1$ is, so $\mathcal{M}$ is necessarily finitely generated over $\Sym(V^0)$. With these assumptions in place, the homological degree in $\Sym(V)$ and in $\mathcal{M}$ is determined by the weight, so the claim follows from the classical theory of the Hilbert function of a finitely generated graded module over a graded polynomial ring \cite[Theorem 11.1]{AtiyahMacdonald}.
\end{proof}

We pair this result with a criterion for detecting finite generation over symmetric algebras with generators of fixed slope.

\begin{lemma}\label{lem:finitely generated rays}
Let $V$ be a finite dimensional bigraded vector space of positive weight. Let $\mathcal{M}$ be a finitely generated $\Sym(V)$-module. If $m_{\max}(V^0)\leq C$, then every ray in $\mathcal{M}$ of slope $C$ is finitely generated over $\Sym(V^0_C)$.  If $m_{\max}(V^0)<C$ then every ray in $\mathcal{M}$ of slope $C$ is eventually $0$. 
\end{lemma}
\begin{proof}
If $m_{\max}(V^0)<C$, then $V_C^0=0$ by assumption, so $\Sym(V_C^0)=\mathbb{Q}$, the monoidal unit. A finitely generated $\mathbb{Q}$-module is simply a finite dimensional bigraded vector space, so the first claim implies the second.

For the first claim, note that $\cM$ is a quotient of a finite sum of shifts of the free module $\Sym(V)$. Hence any slope $C$ ray in $\cM$ is a quotient of a finite sum of shifts of slope $C$ rays in $\Sym(V)$. Since finite sums, shifts, and quotients  preserve finite generation, it suffices to consider the case $\cM = \Sym(V)$.   As in the proof of Lemma \ref{lem:finitely generated growth}, we may further assume that $V=V^0$. Choosing a bihomogeneous basis $\{u_1, \dots, u_m\}$ for $V_{<C}$, we observe that $\mathcal{M}$ is a free module over $\Sym(V^0_C)$ with basis given by the set of monomials of the form $u_1^{j_1} \dots u_m^{j_m}$ with $j_1, \dots, j_m \in \bbZ_{\geq0}$. Fixing a ray $\mathcal{R}=\bigoplus_{t\geq0} \mathcal{M}_{at, bt+i_0}$ with $b/a=C$ and $(a,b)=1, a > 0$, such a monomial lies in $\mathcal{R}$ if and only if \[\sum_{k = 1}^m j_k (a d(u_k) - bw(u_k))  = ai_0.\] By assumption, we have $ad(u_k)-bw(u_k)<0$ for $1\leq k\leq m$, so only finitely many monomials lie in $\mathcal{R}$. Since the product of a monomial and an element of $\Sym(V_C^0)$ lies in $\mathcal{R}$ if and only if the monomial does, the claim follows.
\end{proof}

At times, an alternative to finite generation will be the relevant property.

\begin{definition}\label{def:finite detection}
Let $V$ be a finite type first-quadrant bigraded vector space and $\mathcal{M}$ a $\Sym(V^\vee)$-module. We say that $\mathcal{M}$ is \emph{finitely detected} if there is a surjective linear map $q:\mathcal{M}\to N$ with finite dimensional target such that, for every $m\in \mathcal{M}\setminus \{0\}$, there is a monomial $p\in \Sym(V)$ with $q(\partial_p\cdot m)\neq 0$.
\end{definition}

Through duality, finite detection is closely related to other notions of finiteness. For simplicity, we do not state the following results in the greatest possible generality.

\begin{lemma}\label{lem:finite detection adjoint}
Let $V$ be a finite dimensional, first-quadrant bigraded vector space. The following are equivalent for a bounded below $\Sym(V^\vee)$-module $\mathcal{M}$.
\begin{enumerate}
\item The $\Sym(V^\vee)$-module $\mathcal{M}$ is finitely detected.
\item The adjoint $\Sym(V)$-comodule structure on $\mathcal{M}$ is finitely cogenerated.
\item The dual adjoint $\Sym(V^\vee)$-module structure on $\mathcal{M}^\vee$ is finitely generated.
\end{enumerate}
\end{lemma}
\begin{proof}
Let $q:\mathcal{M}\to N$ be a surjection with finite dimensional target. The composite \[\mathcal{M}\xrightarrow{\delta} \Sym(V)\otimes \mathcal{M}\xrightarrow{\id\otimes q} \Sym(V)\otimes N\] is given by the formula $m\mapsto \sum_p p\otimes q(\partial_p\cdot m)$, where $p$ ranges over the monomial basis for $\Sym(V)$, which is finite in each tridegree. Since $q$ cogenerates if and only if this composite is injective, the implication (1)$\iff$(2) follows. Since the dual of an injection is surjective, and vice versa, we also have the implication (1)$\iff$(3).
\end{proof}

It follows that finitely detected modules enjoy many of the properties of finitely generated modules.

\begin{lemma}\label{lem:finite detection subquotient}
Let $V$ be a finite dimensional, first-quadrant bigraded vector space. Any submodule or quotient of a finitely detected $\Sym(V^\vee)$-module is also finitely detected.
\end{lemma}
\begin{proof}
The claim follows from Lemma \ref{lem:finite detection adjoint} and Noetherianity of $\Sym(V^\vee)$.
\end{proof}

\begin{lemma}\label{lem:dual module growth}
Let $V$ be a finite dimensional, first-quadrant bigraded vector space of slope $C\in\mathbb{Q}$ and $\mathcal{M}$ a finitely detected $\Sym(V^\vee)$-module. The graded dimension of any ray in $\mathcal{M}$ of slope $C$ is eventually equal to a quasi-polynomial of degree at most $\dim V^0-1$ and period dividing $\mathrm{lcm}(V^0)$.
\end{lemma}
\begin{proof}
The claim follows from Lemma \ref{lem:finitely generated growth} and \ref{lem:finite detection adjoint} after negating bidegrees.
\end{proof}

\section{Transits and transit algebras}\label{section:transits and transit algebras}

In this section, we introduce a family of algebras acting on the homology of a Lie algebra.

\subsection{Lie algebras and transits} 

Throughout, the term ``Lie algebra'' refers to a Lie algebra in the category of bigraded vector spaces detailed above. Explicitly, a Lie algebra is a bigraded vector space $\mathfrak{g}$ equipped with a map of bigraded vector spaces $[-,-]:\mathfrak{g}\otimes\mathfrak{g}\to\mathfrak{g}$, called the bracket of $\mathfrak{g}$, satisfying the equations \begin{enumerate}
\item $[x,y]+(-1)^{d(x)d(y)}[y,x]=0$
\item $(-1)^{d(x)d(z)}[[x,y],z]+(-1)^{d(y)d(x)}[[y,z],x]+(-1)^{d(z)d(y)}[[z,x],y]=0$
\end{enumerate} for bihomogeneous elements $x,y,z\in\mathfrak{g}$. A map of Lie algebras is a map of bigraded vector spaces intertwining the respective brackets.

\begin{example}
We maintain the notation of Example \ref{example:degrees}. If $i$ is even, the free Lie algebra on $\mathbb{Q}(n,i)$ is simply $\mathbb{Q}(n,i)$ with trivial bracket. If $i$ is odd, the free Lie algebra on $\mathbb{Q}(n,i)$ is $\mathbb{Q}(n,i)\oplus \mathbb{Q}(2n,2i)$, with bracket given by the isomorphism $\mathbb{Q}(n,i)\otimes\mathbb{Q}(n,i) \cong\mathbb{Q}(2n,2i).$
\end{example}

We emphasize that, while weight is additive under the bracket, all signs are independent of weight. We write $\mathfrak{z}=\mathfrak{z}(\mathfrak{g})$ for the center of the Lie algebra $\mathfrak{g}$ and $\mathfrak{a}=\mathfrak{a}(\mathfrak{g})$ for its Abelianization.

We now formulate the definition from which all stability phenomena considered herein arise. Although the concept is simple, it appears to be new.

\begin{definition}
Let $\mathfrak{g}$ be a Lie algebra. A \emph{transit} of $\mathfrak{g}$ is a pair of maps of Lie algebras
\[\mathfrak{h}\xrightarrow{f}\mathfrak{g}\xrightarrow{g}\mathfrak{k}\] with $\mathfrak{h}$ and $\mathfrak{k}$ Abelian and $f$ central. We say the transit is \emph{null} or \emph{split} if $gf$ is trivial or bijective, respectively, and \emph{exact} if $f$ and $g$ form a short exact sequence. A \emph{map of transits} from $(f_1,g_1)$ to $(f_2,g_2)$ is a commutative diagram of Lie algebras of the form \[\xymatrix{
\mathfrak{h}_1\ar[d]\ar[dr]^-{f_1}\\
\mathfrak{h}_2\ar[r]_-{f_2}&\mathfrak{g}\ar[dr]_-{g_1}\ar[r]^-{g_2}&\mathfrak{a}_2\ar[d]\\
&&\mathfrak{a}_1.
}\]
\end{definition}

Maps of transits compose in the obvious way, forming a category. Note that a Lie algebra admits an exact transit if and only if it is two-step nilpotent.

\begin{example}\label{example:universal transit}
The pair $\mathfrak{z}\to \mathfrak{g}\to \mathfrak{a}$ given by inclusion and projection is a transit, called the \emph{universal transit}.
\end{example}

The term ``universal'' refers to the property of being a terminal object in the category of transits.

\begin{example}\label{example:universal transit}
Given transits $(f_1,g_1)$ and $(f_2, g_2)$ of $\mathfrak{g}$, the \emph{product transit} is the pair \[\mathfrak{h}_1\times \mathfrak{h_2}\xrightarrow{f_1+f_2}\mathfrak{g}\xrightarrow{(g_1,g_2)}\mathfrak{k}_1\times \mathfrak{k}_2\] (we use that the $f_i$ are central). We say the product is \emph{clean} if $g_2f_1=g_1f_2=0$.
\end{example}

\begin{lemma}\label{lem:product decomposition}
Every transit is isomorphic (non-canonically) to the clean product of a split transit and a null transit.
\end{lemma}
\begin{proof}
Given a transit $(f,g)$ of $\mathfrak{g}$, set $\mathfrak{h}_1=\mathfrak{k}_1=\mathrm{im}(gf)$, $\mathfrak{h}_2=\ker(gf)$, and $\mathfrak{k}_2=\mathrm{coker}(gf)$. Choosing splittings, which are splittings of Lie algebras by Abelianness, we have the product decomposition \[\mathfrak{h}_1\times\mathfrak{h}_2\cong\mathfrak{h}\xrightarrow{f} \mathfrak{g}\xrightarrow{g} \mathfrak{k}\cong\mathfrak{k}_1\times\mathfrak{k}_2.\] To conclude the argument, we consider the following four composites:
\begin{align*}
\mathfrak{h}_1&\to \mathfrak{g}\to \mathfrak{k}_2\\
\mathfrak{h}_2&\to \mathfrak{g}\to \mathfrak{k}_1\\
\mathfrak{h}_2&\to \mathfrak{g}\to \mathfrak{k}_2\\
\mathfrak{h}_1&\to \mathfrak{g}\to \mathfrak{k}_1
\end{align*} The first is the restriction to $\mathrm{im}(gf)$ of the quotient map to $\coker(gf)$; the second is the restriction of $gf$ to $\ker(gf)$; and the third is the composite of the previous two. Since all three vanish, it follows that the product in question is clean and that the transit represented by the third composition is null. Since the fourth composite is the identity, the corresponding transit is split, and the proof is complete.
\end{proof}

\subsection{Transit algebras} The purpose of this section is to associate to each transit a naturally occurring algebra. As a matter of notation, we write $\langle -,-\rangle$ for the graded commutator of an algebra $\mathcal{A}$, i.e., $\langle x,y\rangle=xy-(-1)^{d(x)d(y)}yx$ for bihomogeneous elements $x,y\in\mathcal{A}$.

\begin{definition}\label{def:transit algebra}
The \emph{transit algebra} associated to the transit $\mathfrak{h}\xrightarrow{f}\mathfrak{g}\xrightarrow{g}\mathfrak{k}$ is the quotient \[W(f,g)=T(\mathfrak{h}[1]\oplus \mathfrak{k}[1]^\vee)/I,\] where $T$ denotes the tensor algebra and $I$ the two-sided ideal generated by the relations \begin{align*}\langle x,y\rangle&=\langle \lambda,\mu\rangle=0\\
\langle \lambda,x\rangle&=\lambda(g(f(x)))
\end{align*}
for $x,y\in \mathfrak{h}[1]$ and $\lambda,\mu\in \mathfrak{k}[1]^\vee$.
\end{definition}

Although the definition of $W(f,g)$ depends only on the composite $gf$, this algebra interacts in an important way with the Lie algebra $\mathfrak{g}$, as we show in the next section.

The transit algebra extends in the obvious way to a functor from transits to algebras. In particular, given an arbitrary transit $(f,g)$, there is a canonical map of algebras \[W(f,g)\to W(\mathfrak{g}),\] where $W(\mathfrak{g})$ is the transit algebra of the universal transit.

\begin{remark}
The transit algebra is functorial for maps between transits of \emph{different} Lie algebras, in the sense of the following commutative diagram:
\[
\xymatrix{
\mathfrak{h}\ar[d]\ar[r]&\mathfrak{g}\ar[d]\ar[r]&\mathfrak{k}\\
\mathfrak{h}'\ar[r]&\mathfrak{g}'\ar[r]&\mathfrak{k}'.\ar[u]
}
\] Note that this diagram specializes to our earlier definition of a map of transits when the middle arrow is the identity. We make no use of this extended functoriality.
\end{remark}

\begin{example}\label{example:null algebra}
If $(f,g)$ is null, then $W(f,g)=\Sym(\mathfrak{h}[1]\oplus \mathfrak{k}[1]^\vee)$.
\end{example}

\begin{example}\label{example:split algebra}
If $(f,g)$ is split, then $W(f,g)$ is isomorphic to the Weyl algebra on the vector space $\mathfrak{h}[1]$.
\end{example}

\begin{lemma}\label{lem:transit product}
The transit algebra of a clean product is canonically isomorphic to the tensor product of the transit algebras of the factors.
\end{lemma}
\begin{proof}
Cleanness implies that $\langle \lambda, x\rangle=0$ for $x\in \mathfrak{h}_i$ and $\lambda\in \mathfrak{k}_j$ when $i\neq j\in \{1,2\}$, and the claim follows.
\end{proof}

We record the following simple result, although we make no use of it in what follows.

\begin{proposition}\label{prop:noetherian}
Let $\mathfrak{h}\xrightarrow{f}\mathfrak{g}\xrightarrow{g}\mathfrak{k}$ be a transit with $\mathfrak{h}$ and $\mathfrak{k}$ finite dimensional. The transit algebra $W(f,g)$ is Noetherian.
\end{proposition}
\begin{proof}
By Lemmas \ref{lem:product decomposition} and \ref{lem:transit product} and Examples \ref{example:null algebra} and \ref{example:split algebra}, it suffices to show that the tensor product of a finitely generated symmetric algebra and a finitely generated Weyl algebra is Noetherian. Both algebras are examples of skew polynomial rings in the sense of \cite[Thm. 2.9]{McConnellRobson:NNR}, which implies that iterated skew polynomial rings over a Noetherian base ring (such as $\mathbb{Q}$) are Noetherian.
\end{proof}

\subsection{Action on homology} The role of the transit algebras introduced above is as a source of extra structure on the homology of Lie algebras.

\begin{definition}\label{def:liehomology}
The \emph{Chevalley--Eilenberg complex} of the Lie algebra $\mathfrak{g}$ is the bigraded vector space $\mathrm{CE}(\mathfrak{g})=\Sym(\mathfrak{g}[1])$ equipped with the differential $\partial$ determined as a coderivation by the formula \[\partial( x y)=(-1)^{d(x)} [x,y].\]
The \emph{Lie algebra homology} of $\mathfrak{g}$ is the bigraded vector space $H_*^\mathrm{Lie}(\mathfrak{g})=H_*(\mathrm{CE}(\mathfrak{g}), \partial)$.
\end{definition}

The reader interested in an explicit formula for $\partial$ may consult \cite[Ch. 22]{FelixHalperinThomas:RHT}, for example. Since it plays no direct role in our arguments, we omit it as an unnecessary distraction.

\begin{remark}
Our notation is abusive in that the symbols $x$, $y$, and $[x,y]$ refer to the corresponding elements of $\mathfrak{g}[1]$, while $d(x)$ is the degree of $x$ as an element of $\mathfrak{g}$.
\end{remark}

The implicit claim that $\partial^2=0$ is equivalent to the Jacobi identity (2) above. Note that $\partial\equiv0$ if $\mathfrak{g}$ is Abelian, whence $\mathrm{CE}(\mathfrak{g})=\Sym(\mathfrak{g}[1])$ as a chain complex.

\begin{theorem}\label{thm:transit action}
Let $\mathfrak{g}$ be a Lie algebra, and suppose that $\mathfrak{g}$ is bounded below of finite type as a bigraded vector space. For any transit $(f,g)$ of $\mathfrak{g}$, there is a canonical, functorial action by chain maps of $W(f,g)$ on $\mathrm{CE}(\mathfrak{g})$. In particular, there is a canonical, functorial action on $H_*^\mathrm{Lie}(\mathfrak{g})$.
\end{theorem}
\begin{proof}
It suffices to consider the universal case. By Abelianness, addition equips $\mathfrak{z}$ with the structure of a commutative monoid with respect to the Cartesian monoidal structure on Lie algebras. For the same reason, the composite map \[\mathfrak{z}\times\mathfrak{g}\subseteq \mathfrak{g}\times\mathfrak{g}\xrightarrow{+} \mathfrak{g}\] is a map of Lie algebras, equipping $\mathfrak{g}$ with the structure of a $\mathfrak{z}$-module. The Chevalley--Eilenberg complex is a symmetric monoidal functor, so we obtain an action of $\mathrm{CE}(\mathfrak{z})=\Sym(\mathfrak{z}[1])$ on $\mathrm{CE}(\mathfrak{g})$, hence on its homology (we again use that $\mathfrak{z}$ is Abelian).

On the other hand, the diagonal equips any Lie algebra with the structure of cocommutative comonoid with respect to the Cartesian monoidal structure, and the composite map \[\mathfrak{g}\xrightarrow{\Delta} \mathfrak{g}\times\mathfrak{g}\xrightarrow{\pi\times\id}\mathfrak{a}\times\mathfrak{g}\] equips $\mathfrak{g}$ with the structure of an $\mathfrak{a}$-comodule, where $\pi$ is the projection to the quotient. Applying the Chevalley--Eilenberg complex and dualizing, we obtain by adjunction an action of $\Sym(\mathfrak{a}[1]^\vee)\cong \mathrm{CE}(\mathfrak{a})^\vee$ on $\mathrm{CE}(\mathfrak{g})$, hence on its homology (we use that $\mathfrak{a}$ is Abelian and $\mathfrak{g}$ bounded below of finite type).

It remains to check that these two actions descend to an action of $W(\mathfrak{g})$, for which it suffices to verify the last relation. To begin, we note that, at the level of underlying graded objects, the first action is given by the composite \[\Sym(\mathfrak{z}[1])\otimes \Sym(\mathfrak{g}[1])\subseteq \Sym(\mathfrak{g}[1])\otimes \Sym(\mathfrak{g}[1])\xrightarrow{m}\Sym(\mathfrak{g}[1]),\] where $m$ is the usual multiplication (note that $m$ is not itself a chain map), while the second action is given by the composite (cap product) \[\Sym(\mathfrak{a}[1]^\vee)\otimes\Sym(\mathfrak{g}[1])\xrightarrow{\Sym(\pi^\vee)\otimes\delta} \Sym(\mathfrak{g}[1]^\vee)\otimes \Sym(\mathfrak{g}[1])\otimes \Sym(\mathfrak{g}[1])\xrightarrow{\mathrm{ev}\cdot\id} \Sym(\mathfrak{g}[1]),\] where $\delta$ is the usual comultiplication.

Choose a bihomogeneous basis for $\mathfrak{g}$, and fix nonzero bihomogeneous elements $x\in \mathfrak{z}[1]$ and $\lambda\in\mathfrak{a}[1]^\vee$ and a monomial $p\in\Sym(\mathfrak{g}[1])$. Writing $\delta(p)=\sum_i p_i\otimes p_i'$, and abusively identifying $\lambda$ and $\Sym(\pi^\vee)(\lambda)$, we calculate from the definitions that \begin{align*}
\langle \lambda, x\rangle \cdot p&=\lambda\cdot (x p)-(-1)^{d(\lambda)d(x)}x(\lambda\cdot p)\\
&=(\lambda\otimes 1)(\delta(xp))-(-1)^{d(\lambda)d(x)}x(\lambda\otimes 1)(\delta(p))\\
&=\sum_i\left[(\lambda\otimes 1)\left(xp_i\otimes p_i'+(-1)^{d(x)d(p_i)}p_i\otimes xp_i'\right)-(-1)^{d(\lambda)d(x)}x(\lambda\otimes 1)(p_i\otimes p_i') \right]\\
&=\sum_i\left[\lambda(xp_i) p_i'+(-1)^{d(x)d(p_i)}\lambda(p_i) xp_i'-(-1)^{d(\lambda)d(x)}\lambda(p_i)xp_i'\right]\\
&=\lambda(x)p+\sum_i\left[(-1)^{d(x)d(p_i)}\lambda(p_i) xp_i'-(-1)^{d(\lambda)d(x)}\lambda(p_i)xp_i'\right]\\
&=\lambda(x)p+\sum_{i: d(\lambda)+d(p_i)=0}\left[(-1)^{d(x)d(p_i)}\lambda(p_i) xp_i'-(-1)^{d(\lambda)d(x)}\lambda(p_i)xp_i'\right]\\
&=\lambda(x)p,
\end{align*} where we use that $\lambda$ vanishes on monomials of polynomial degree different from $1$ and on those of bidegree different from $(-w(\lambda),-d(\lambda))$. Since $p$ was arbitrary, this calculation establishes the desired relation.
\end{proof}

\section{Stability phenomena}

In this section, we give examples of stability phenomena in Lie algebra homology arising from the transit algebra action introduced above. We then deduce our results about the homology of configuration spaces.

\subsection{Transit algebras and stability}

The purpose of this section is to demonstrate that the transit algebra actions described in Theorem \ref{thm:transit action} give rise to stability phenomena in the homology of Lie algebras. For simplicity, we work under the following standing assumptions.

\begin{assumption}\label{assumption:lie}
The Lie algebra $\mathfrak{g}$ is finite dimensional with $\mathfrak{g}[1]$ first-quadrant and connected.
\end{assumption}

There are two versions of the following result. We state the version for the maximal slope; the version for the minimal slope is obtained by reversing inequalities and interchanging $\min$ and $\max$.  Throughout, a \emph{ray} in $H_*(\fg)$ is with respect to the bigraded structure of Definition \ref{def:liehomology}.

\begin{proposition}
	\label{prop:lie stability}
Let $g: \fg \to \fk$ be an Abelian quotient, and suppose that we have the inequalities \begin{align*}
m_{\max}(\fk[1]^0)&\leq C\\
m_{\max}(\ker g[1]^0)&<C.
\end{align*} Each ray of slope $C$ in $H_*^\mathrm{Lie}(\mathfrak{g})$ is finitely detected over $ \Sym(\mathfrak{k}^\vee[-1]^0_C ) \subseteq W(0,g).$ In particular, the graded dimension of such a ray is eventually equal to a quasi-polynomial of degree at most $\dim (\mathfrak{k}[1]^0_C)-1$ and period dividing $\mathrm{lcm}(\mathfrak{k}[1]^0_C)$.
\end{proposition}
\begin{proof}
Any ray of slope $C$ in $H_*^\mathrm{Lie}(\mathfrak{g})$ is a $\Sym(\mathfrak{k}^\vee[-1]_C^0)$-submodule and a subquotient of the corresponding ray in $\mathrm{CE}(\mathfrak{g})$; therefore, by Lemmas \ref{lem:finite detection adjoint}, \ref{lem:finite detection subquotient}, and \ref{lem:dual module growth}, it suffices to show that every ray of slope $C$ in the dual adjoint module $\mathrm{CE}(\mathfrak{g})^\vee$ is finitely generated.

The action of $\Sym(\mathfrak{k}^\vee[-1]^0_C)$ on $\rC\rE(\fg)^\vee$, which respects the differential, extends along the monomorphism $g^\vee$ to an action of $\Sym(\mathfrak{g}^\vee[-1])$, which does not respect the differential.  Under this action, $\mathrm{CE}(\mathfrak{g})^\vee$ is free of rank $1$, hence finitely generated. Since $m_{\max}(\mathfrak{g}^\vee[-1]^0)\leq C$, it follows from Lemma \ref{lem:finitely generated rays} that each ray of slope $C$ is finitely generated over $\Sym(\mathfrak{g}^\vee[-1]_C^0)$.  Our assumption implies that $\mathfrak{g}[1]^0_C\cap \ker g [1]^0=0$, so $\mathfrak{g}^\vee[-1]^0_C=\mathfrak{k}^\vee[-1]^0_C$.    
\end{proof}

In a sense, the next result is dual to Proposition \ref{prop:lie stability}. Again, there are two versions, only one of which we state.

\begin{proposition}\label{prop:lie extremal}
Let $f: \fh \to \fg$ be a central subalgebra, and suppose that we have the inequalities \begin{align*}
m_{\max}(\fh[1]^0)&\leq C\\
m_{\max}(\coker f[1]^0)&<C.
\end{align*} Each ray of slope $C$ in $H_*^\mathrm{Lie}(\mathfrak{g})$ is finitely generated over $ \Sym(\mathfrak{h}[1]^0_C ) \subseteq W(f,0).$ In particular, the graded dimension of such a ray is eventually equal to a quasi-polynomial of degree at most $\dim (\mathfrak{h}[1]^0_C)-1$ and period dividing $\mathrm{lcm}(\mathfrak{h}[1]^0_C)$.
\end{proposition}

\begin{proof}
Since each ray of slope $C$ in $\mathrm{CE}(\mathfrak{g})$ is a submodule with subquotient the corresponding ray of slope $C$ in $H_*^\mathrm{Lie}(\mathfrak{g})$, it suffices by Noetherianity and Lemma \ref{lem:finitely generated growth} to show that each ray of slope $C$ in $\mathrm{CE}(\mathfrak{g})$ is finitely generated.

The action of $\Sym(\mathfrak{h}[1]_C^0)$, which respects the differential, extends along the inclusion of $\mathfrak{h}$ to an action of $\Sym(\mathfrak{g}[1])$, which does not respect the differential. Under this action, $\mathrm{CE}(\mathfrak{g})$ is free of rank $1$, hence finitely generated. Since $m_{\max}(\mathfrak{g}[1]^0)\leq C$, it follows from Lemma \ref{lem:finitely generated rays} that each ray of slope $C$ is finitely generated over $\Sym(\mathfrak{g}[1]_C^0)$, but $\mathfrak{g}[1]^0_C=\mathfrak{h}[1]^0_C$ by our assumption.  
\end{proof}

The same method yields the following mild extension.

\begin{corollary}\label{cor:semidirect product}
Let $\mathfrak{g}$ be as in Proposition \ref{prop:lie extremal} and consider a semidirect product $\tilde{\mathfrak{g}}=\mathfrak{g}\rtimes \mathfrak{l}$ with $\mathfrak{l}$ free and finitely generated. If $\mathfrak{l}$ centralizes $\mathfrak{h}[1]_C^0$, then the conclusion of Proposition \ref{prop:lie extremal} holds for $\tilde{\mathfrak{g}}$.
\end{corollary}
\begin{proof}
The action of $\Sym(\mathfrak{h}[1]^0_C)$ on $\mathrm{CE}(\mathfrak{g})$ described in the proof of Proposition \ref{prop:lie extremal} extends to an action on the underlying bigraded vector space of $\mathrm{CE}(\tilde{\mathfrak{g}})$, the latter being a sum of bigraded shifts of the former. Our assumption on the relationship between $\mathfrak{l}$ and $\mathfrak{h}$ guarantees that this action is compatible with the differential, and the homological Lyndon--Hochschild--Serre spectral sequence \cite[Ch. 34]{FelixHalperinThomas:RHT} \[H_*^\mathrm{Lie}(\mathfrak{l};H_*^\mathrm{Lie}(\mathfrak{g}))\implies H_*^\mathrm{Lie}(\tilde{\mathfrak{g}})\] is a spectral sequence of $\Sym(\mathfrak{h}[1]^0_C)$-modules.\footnote{For information on Lie algebra homology with coefficients, the reader may consult \cite[Ch. 22]{FelixHalperinThomas:RHT}.} By Noetherianity and Lemma \ref{lem:finitely generated growth}, it suffices to show that each ray of slope $C$ in the initial page of this spectral sequence is finitely generated.

Writing $V$ for the finite dimensional bigraded vector space generating $\mathfrak{l}$ as a free Lie algebra, the page in question is isomorphic to $\mathrm{Tor}^{T(V)}_*(H_*^\mathrm{Lie}(\mathfrak{g}),\mathbb{Q})$, which may be calculated as the homology of the complex \[H_*^\mathrm{Lie}(\mathfrak{g})\otimes V\to H_*^\mathrm{Lie}(\mathfrak{g})\] of $\Sym(\mathfrak{h}[1]^0_C)$-modules. Each ray of slope $C$ in this complex is a finite sum of bigraded shifts of rays of slope $C$ in $H_*^\mathrm{Lie}(\mathfrak{g})$, each of which has already been shown to be finitely generated. By Noetherianity of $\Sym(\mathfrak{h}[1]^0_C)$,  the homology of this complex is finitely generated. This completes the proof.
\end{proof}

\begin{remark}\label{remark:iterated semidirect product}
The same method of proof establishes the analogous conclusion for iterated semidirect products with free Lie algebras. Details are left to the reader.
\end{remark}

Our final corollary uses the action of both halves of the transit algebra to establish that rays are free modules in certain cases.  

\begin{corollary}\label{thm:free}
		Let $\fh \xrightarrow{f} \fg \xrightarrow{g} \fk$  be a split transit, and suppose that we have the inequalities \begin{align*}
m_{\max}(\fh[1]^0)&\leq C\\
m_{\max}(\coker f[1]^0)&<C.
\end{align*} Each ray of slope $C$ of $H^\mathrm{Lie}_*(\fg)$ is a finitely generated free module over $\Sym(\fh[1]^0_C) \subseteq W(f,g)$.
\end{corollary}
\begin{proof}
Finite generation follows from Proposition \ref{prop:lie extremal}. By Example \ref{example:split algebra} and Theorem \ref{thm:transit action}, the action of $\Sym(\fh[1]^0_C)$ extends to an action of the Weyl algebra on $\fh[1]^0_C$. It is well known that a finitely generated module over a polynomial ring whose action extends to the Weyl algebra is free---briefly, choosing a minimal set of generators and a relation $R$ of minimal degree, we note that, unless $R$ is constant (contradicting the first requirement of minimality), the derivative of $R$ is a relation of strictly smaller degree. Here $R$ is an inhomogeneous relation in the generators, and its degree is the maximum of the degrees of its nonzero homogeneous terms, and a constant relation is one of degree $0$.
\end{proof}

\subsection{Application to configuration spaces}\label{section:applications}

Fix a $d$-manifold $M$, and write \[\mathfrak{g}_M=H_c^{*}(M;\mathbb{Q}^w)\otimes v\oplus H_c^*(M;\mathbb{Q})\otimes [v,v],\] where $v$ and $[v,v]$ are formal parameters of bidegree $(1,d-1)$ and $(2,2d-2)$, respectively, and cohomology is regarded as concentrated in negative degrees and weight $0$ (the reader is reminded that $\mathbb{Q}^w$ denotes the orientation sheaf). This bigraded vector space becomes a Lie algebra by declaring that the only nonzero components of the bracket are given by the equation \[[\alpha\otimes v,\beta\otimes v]=(-1)^{d(\beta)(d-1)}\alpha\beta\otimes [v,v],\] where $\alpha$ and $\beta$ are multiplied via the twisted cup product.

\begin{theorem}[{\cite{Knudsen}}]\label{thm:lie homology}
Let $M$ be a manifold of even dimension $d$. There is an isomorphism of bigraded vector spaces \[\bigoplus_{n\geq0} H_*(B_n(M);\mathbb{Q})\cong H_*^\mathrm{Lie}(\mathfrak{g}_M).\]
\end{theorem}

Through Theorem \ref{thm:lie homology}, stability phenomena in Lie algebra homology become stability phenomena in the homology of configuration spaces. As a matter of notation, we write $\mathfrak{h}_M=H_c^*(M;\mathbb{Q})\otimes [v,v]$ and $\mathfrak{k}_M=H_c^*(M;\mathbb{Q}^w)\otimes v$, considered as Abelian Lie algebras. Thus, we have the exact transit $\mathfrak{h}_M\to \mathfrak{g}_M\to \mathfrak{k}_M$. We observe that, for degree $j$ cohomology classes $\alpha \in H^j_c(M, \bbQ)$ and $\alpha' \in H^j_c(M, \bbQ^w)$, as elements of $\mathfrak{g}_M[1]$, we have \begin{align*}
d(\alpha'\otimes v)=m(\alpha'\otimes v)&=d-j\\
d(\alpha\otimes[v,v])&=2d-j-1\\
m(\alpha\otimes [v,v])&=d-\frac{j+1}{2}.
\end{align*} Note that $\mathfrak{g}_M$ satisfies Assumption \ref{assumption:lie}.

As a first example, we give a proof of classical homological stability.

\begin{proof}[Proof of Theorem \ref{thm:homological stability}]
Since $M$ is a manifold, we have $0\leq j\leq d$ in the equations above, so we have the inequalities \begin{align*}
m_{\min}(\mathfrak{h}_M[1])&\geq\frac{d-1}{2}>0\\
m_{\min}(\mathfrak{k}_M[1])&\geq0
\end{align*} (we use that $d>1$). We now invoke the minimal slope version of Proposition \ref{prop:lie stability} with $C=0$. Observing that $\mathfrak{k}_M[1]_0^0\cong H_c^d(M;\mathbb{Q}^w)\cong H_0(M;\mathbb{Q})\neq0$ by Poincar\'{e} duality, the degree bound on the resulting growth quasi-polynomial follows. Since $\mathfrak{k}_M$ is concentrated in weight $1$, this quasi-polynomial is in fact a polynomial.
\end{proof}

We turn now to extremal stability, our main result.

\begin{proof}[Proof of Theorem \ref{thm:extremal stability}]
Assume first that $M$ has no compact component. By Poincar\'{e} duality, we have $H_c^0(M;\mathbb{Q}^w)\cong H_d(M;\mathbb{Q})=0$. Since we likewise have $H_c^0(M;\mathbb{Q})\cong \widetilde{H}^0(M^+;\mathbb{Q})=0$, it follows that $0<j\leq d$ in the above equations, whence
\begin{align*}
m_{\max}(\mathfrak{h}_M[1])&\leq d-1\\
m_{\max}(\mathfrak{k}_M[1])&\leq d-1
\end{align*}
We observe that $d(\alpha\otimes v)$ is even if and only if $j$ is even, so $m_{\max}(\mathfrak{k}_M[1]^0)< d-1$; thus, we may apply Proposition \ref{prop:lie extremal} with $C=d-1$. The degree bound follows as before upon noting that $\mathfrak{h}[1]_{d-1}^0\cong H_c^1(M;\mathbb{Q})\cong H_{d-1}(M;\mathbb{Q}^w)$ by Poincar\'{e} duality, and the period estimate follows from the fact that $\mathfrak{h}_M$ is concentrated in weight $2$.

In general, write $\{M_i\}_{i\in I}$ for the (finite) set of compact components of $M$ and $\dot{M}$ for the manifold obtained by removing a single point from each. We have the exact sequence of Lie algebras \[\mathfrak{g}_{\dot{M}}\to \mathfrak{g}_M\to \prod_{i\in I} \mathfrak{l}_{i},\] where $\mathfrak{l}_{i}$ is free on one generator of bidegree $(1,d-1)$ or $(2,2d-2)$ according to whether $M_i$ is orientable. This sequence splits, expressing $\mathfrak{g}_M$ as an iterated semidirect product of $\mathfrak{g}_{\dot{M}}$ with free Lie algebras. Since the $\mathfrak{l}_i$ centralize $\mathfrak{h}_M=H_c^*(M;\mathbb{Q})\otimes[v,v]$, the claim follows from Corollary \ref{cor:semidirect product} and Remark \ref{remark:iterated semidirect product}.
\end{proof}

We note the following structural statement, which we have established in the proof of Theorem \ref{thm:extremal stability}.

\begin{theorem}\label{thm:structural}
		Let $M$ be as in Theorem \ref{thm:extremal stability}. Then every ray of slope $d-1$ in $\bigoplus_{n\geq0} H_*(B_n(M))$ is a finitely generated module over $\Sym(H_{d-1}(M; \bbQ^w) \langle 2 \rangle [2d-2])$ (the free commutative algebra on $H_{d-1}(M, \bbQ^w)$, with generators in bi-degree $(2,2d-2)$).    If  $H_d(M; \bbQ) = 0$, then these  modules are free.  
\end{theorem}
\begin{proof}
	The proof of Theorem \ref{thm:extremal stability} uses the central subalgebra  $H^1_c(M) \otimes [v,v] \subseteq \fh_M \subseteq \fg_M $.   By Poincar\'e duality $ H^1_c(M) \iso H_{d-1}(M; \bbQ^w)$.   It is shown that every slope $d-1$ ray is a finitely generated $\Sym(H^1_c(M)\langle 2 \rangle [1])$ module. When $H^0_c(M; \bbQ^w) \iso H_{d}(M ;\bbQ) = 0,$ we have that  $$\fj_M := \bigoplus_{j>0}H^j_c(M; \bbQ^w) \otimes v  \oplus \bigoplus_{j \neq 1}  H^j_c(M; \bbQ) \otimes [v,v]$$  is an ideal--- in fact $[\fg_M, \fg_M] \subseteq \fj_M$  since any element of $\fg_M$ of the form $\alpha \otimes [v,v]$  is central,  and any pair of elements $\beta_1 \otimes v, \beta_2 \otimes v \in \fg_M$ with $d(\beta_i) > 0$  satisfies $[\beta_1 \otimes v, \beta_2 \otimes v] \in \fj_M$.

 Further,  $H^1_c(M) \otimes [v,v] \to \fg_M \to \fg_M/\fj_M$  is a split transit.  Hence as in the proof of Corollary \ref{thm:free},  the action of $\Sym(H_{d-1}(M; \bbQ^w)[1]\langle 2 \rangle)$ extends to an action of the Weyl algebra.  Therefore every ray of slope $d-1$  is a finitely generated free module.  
\end{proof}

\begin{example}
The assumption that $H_d(M; \bbQ) = 0$ is needed to ensure that the modules are free. For example, let $\Sigma$ be a compact orientable surface of genus $1$. If $\bigoplus_{n \geq 0} H_n(B_n(\Sigma);\bbQ) $ were free, then the sequence $\dim H_{2n}(B_{2n}(\Sigma);\bbQ)$ would agree with a polynomial of degree $1$. However, it follows from \cite[Corollary 4.6]{DK} that $\dim H_{0}(B_{0}(\Sigma);\bbQ)=1$, $\dim H_{2}(B_{2}(\Sigma);\bbQ)=1$, but $\dim H_{4}(B_{4}(\Sigma);\bbQ)=4$.

\end{example}

\subsection{Vanishing and a loose end}

If $H_{d-1}(M;\mathbb{Q})=0$, then the conclusion of Theorem \ref{thm:extremal stability} is that $H_{\nu_n-i}(B_n(M);\mathbb{Q})=0$ for $n$ large, which is to say that every ray of slope $d-1$ eventually vanishes. In such a situation, the question of extremal stability should concern rays of smaller slope. For simplicity, we assume that $M$ is orientable.

\begin{theorem}\label{thm:vanishing}
Let $M$ be an orientable manifold of even dimension $d\geq 2$. Fix $r\geq3$, and suppose that $H_{d-s}(M;\mathbb{Q})=0$ for $0< s<r$. Fix $i \in \bbZ$. For $n$ sufficiently large, the function $n\mapsto \dim H_{ n ( d - 1 - \lfloor\frac{r}{2}\rfloor)-i}(B_n(M);\mathbb{Q})$
%\note{Ben: I got confused about the right formula for codimension $i$ in this situation} 
%\note{Phil: Highest nonvanishing odd group is $H^{d-r}$ if $r$ odd and $H^{d - (r+1)}$ if $r$ even.  So we get slope $(d-1 + (d-r) )/2$ if $r$ odd and $(d-1 + (d-(r+1)) )/2$ if $r$ even.  I think this is $d - 1 -  \lfloor \frac{r}{2}\rfloor.$ } 
is equal to a quasi-polynomial in $n$ of period dividing two and degree at most ${\dim H_{d- r}(M;\mathbb{Q})-1}$ if $r$ is odd and ${\dim H_{d- r-1}(M;\mathbb{Q})-1}$  if $r$ is even.  
%\note{Phil: I also changed the other red parts} 
\end{theorem}
\begin{proof} The proof mirrors that of Theorem \ref{thm:extremal stability}, with the reduction to the non-compact case proceeding unchanged. We aim to apply Proposition \ref{prop:lie extremal} with $C=d-1-\lfloor\frac{r}{2}\rfloor$. Since $M$ is orientable, twisted and untwisted cohomology coincide, and $H_c^s(M;\mathbb{Q})\cong H_{d-s}(M;\mathbb{Q})=0$ for $0\leq s< r$. Thus, we have
\begin{align*}
m_{\max}(\mathfrak{h}_M[1])&\leq d-\frac{r+1}{2}\\
m_{\max}(\mathfrak{k}_M[1])&\leq d-r<C,
\end{align*}
where we use that $r\geq 3$. We observe that $d(\alpha\otimes[v,v])$ is even if and only if $\alpha$ has odd degree, so \[m_{\max}(\mathfrak{h}_M[1]^0)\leq \begin{cases}
d-\frac{r+1}{2}&\quad r\text{ odd}\\
d-\frac{r+2}{2}&\quad r\text{ even},
\end{cases}\] as desired. As before, the claim follows from Poincar\'{e} duality.
\end{proof}

In the case $r=2$, the situation is unclear, and we ask the following.

\begin{question}\label{quest:requals1}
Suppose that $H_{d-1}(M;\mathbb{Q}) = 0$. For $i\in\bbZ$, is the Hilbert function $$n \mapsto \dim H_{{n(d-2)-i}}(B_n(M);\mathbb{Q})$$ 
eventually a quasi-polynomial?
\end{question}

This question has an affirmative answer under the further assumption that $H_{d-2}(M;\mathbb{Q})=0$, by Theorem \ref{thm:vanishing}, and it is not hard to show that the same is true when $H_{d-3}(M;\mathbb{Q})=0$  using Proposition \ref{prop:lie stability}.

\begin{remark}
Since the initial draft of this paper appeared, Yameen has answered Question \ref{quest:requals1} in the affirmative in the case of $M=\mathbb{CP}^m$ \cite{Yameen}. 
\end{remark}

\subsection{The odd dimensional case}\label{section:odd case} We comment briefly on the case of $d$ odd, which is encompassed implicitly by the results of \cite{BCT}. 

In odd dimensions, the corresponding Lie algebra provided by \cite{Knudsen} is the Abelian Lie algebra $H_c^*(M;\mathbb{Q}^w)\otimes v$, so \[\bigoplus_{n\geq0}H_*(B_n(M);\mathbb{Q})\cong \Sym(H_*(M;\mathbb{Q}))\] by Poincar\'{e} duality (as was originally proven in \cite{BCT}). It follows easily that every ray of slope zero is free and finitely generated over $\Sym(H_0(M))$, so $\dim H_i(B_n(M);\mathbb{Q})$ is a polynomial in $n$ of exact degree $\dim H_0(M;\mathbb{Q})-1$. Thus, Theorem \ref{thm:homological stability} holds in this case. As for Theorem \ref{thm:extremal stability}, the corresponding statement is that every ray of slope $d-1$ is free and finitely generated over $\Sym(H_{d-1}(M;\mathbb{Q}))$, so $\dim H_{\nu_n-i}(B_n(M);\mathbb{Q})$ is a polynomial in $n$ of exact degree $\dim H_{d-1}(B_n(M);\mathbb{Q})-1$.

\bibliographystyle{amsalpha}
\bibliography{Extreme}

\end{document}